\documentclass[11pt,english]{smfart}

\usepackage{eucal,enumerate,dsfont,txfonts,amssymb}

\usepackage[text={6.5in,9in},centering]{geometry} 
\usepackage{dsfont,eucal,mathptmx,txfonts}

\usepackage[dvipsnames]{xcolor}   
\usepackage{xparse}
\usepackage{xr-hyper}
\definecolor{brightmaroon}{rgb}{0.76, 0.13, 0.28}
\usepackage[linktocpage=true,colorlinks=true,hyperindex,citecolor=teal,linkcolor=blue]{hyperref} 

\newtheorem{theorem}{Theorem}[section] 
\newtheorem{proposition}[theorem]{Proposition}
\newtheorem{lemma}[theorem]{Lemma}
\newtheorem{corollary}[theorem]{Corollary} 
\theoremstyle{definition}
\newtheorem{definition}[theorem]{Definition} 
\newtheorem{example}[theorem]{Example} 
\newtheorem{examples}
[theorem]{Examples} 
\newtheorem{remark}[theorem]{Remark}  
\newtheorem{notation}[theorem]{Notation}       

\newcommand{\inv}{^{\raisebox{.2ex}{$\scriptscriptstyle -1$}}}   
 
\usepackage{a4wide}

\usepackage[all]{xy}

\begin{document} 

\title{Ideal spaces: An extension of structure spaces of rings}    
   
\author{Themba Dube}
\address{[1] Department of Mathematical Sciences, University of South Africa, P.O. Box 392, 0003 Pretoria, South Africa, [2] National Institute for Theoretical and Computational Sciences (NITheCS), South Africa.}
\email{dubeta@unisa.ac.za}

\author{Amartya Goswami} 
\address{[1] Department of Mathematics and Applied Mathematics,
University of Johannesburg, 
P.O. Box 524, 
Auckland Park 
2006, 
South Africa, [2] National Institute for Theoretical and Computational Sciences (NITheCS), South Africa.}

\email{agoswami@uj.ac.za}






 \maketitle
   
\begin{abstract}
Can there be a structure space-type theory for an arbitrary class of ideals of a ring? The ideal spaces introduced in this paper allows such a study and our theory includes (but not restricted to) prime, maximal, minimal prime, strongly  irreducible, irreducible, completely irreducible, proper, minimal, primary, nil, nilpotent, regular, radical, principal, finitely generated ideals. We characterise ideal spaces that are sober. We introduce the notion of  a strongly disconnected spaces and show that for a ring with zero Jacobson radical, strongly disconnected ideal spaces containing all maximal ideals of the ring imply existence of non-trivial idempotent  elements in the ring. We also give a sufficient condition for a spectrum to be connected. 
\end{abstract}

\section*{Introduction}

Stone introduced \cite{S37} a topology on the set of prime ideals of a Boolean ring. That topology is induced by a (Kuratowski) closure operator defined in terms of intersection  and inclusion relations  among ideals of the Boolean ring. These two relations respectively give kernel and hull maps (see Definition \ref{hkm}). The topology so obtained is now known as the Stone topology or the hull-kernel topology. In the special cases of rings of continuous functions and commutative normed rings, it has also been introduced on maximal ideals by Gelfand and Kolmogoroff  \cite{GK39} (see also \cite{GJ60,S39}) and by Gelfand and \v{S}ilov \cite{GS41} respectively. For the algebra of all continuous complex-valued functions, Loomis \cite{L53} considered the same topology on maximal ideals. 

Jacobson \cite{J45} (see also \cite{J56}) showed that the set of primitive ideals of an arbitrary ring may be  endowed with a hull-kernel topology and he called the corresponding topological space the structure space of the ring. It was observed by McCoy \cite{M49} that the same topology may be considered on the set of generalised prime ideals defined therein.   
For the construction of an affine scheme,  Grothendieck \cite{G60} considered it on prime ideals of a commutative ring, and called it the Zariski topology or the spectral topology. In \cite{HJ65} (see also \cite{H71}), Henriksen and Jerison  studied the hull-kernel topology on minimal prime ideals of a commutative ring. In a more recent paper, Azizi \cite{A08} studied strongly irreducible ideals of a commutative ring endowed with the hull-kernel topology.  A study of topologies on ideals and proper ideals of an $R$-module can be found in \cite{FFS16}, whereas primary ideals and principal ideals with constructible topologies have been studied in \cite{FS20}. In \cite{O16}, it has been shown that with the involvement of Zariski topology, every ideal has an irredundant representation in terms of irreducible ideals. Apart from studying properties of a hull-kernel topology on various settings as mentioned above, Hochster \cite{H67} (see also \cite{H69, DST19})  proved that the structure space of prime ideals of a commutative ring is spectral. 

In order to `unify' all the above-mentioned works, in a general sense we may say that a structure space is a topological space whose underlying set is a class of ideals endowed with a hull-kernel topology. Moreover, we observe that irrespective of the types of rings, the choices of classes of ideals must be such that we can define hull-kernel topologies on them. McKnight \cite{M53} characterised such classes of ideals, and a necessary and sufficient condition (\ref{mic}) that a class of ideals must satisfy is nothing but a restricted version of the definition of a strongly irreducible ideal (see Theorem \ref{hkt}). 
 
This paper is motivated by the possibility of extending to all classes of ideals some of the results obtained for structure spaces. Before we elaborate on this point, let us first have a working terminology. By a \emph{spectrum} of $R$, we mean a class of same  `type' of ideals.
Here are examples of some well-known spectra: prime ideals ($\mathrm{Spec}(R)$),  
maximal ideals ($\mathrm{Max}(R)$), proper ideals 
($\mathrm{Prp}(R)$),  radical ideals 
($\mathrm{Rad}(R)$),  minimal ideals
($\mathrm{Min}(R)$),  
minimal prime ideals ($\mathrm{Spn}(R)$), primary ideals
($\mathrm{Prm}(R)$), 
nil ideals
($\mathrm{Nil}(R)$),  nilpotent ideals
($\mathrm{Nip}(R)$),  irreducible ideals
($\mathrm{Irr}(R)$),   completely irreducible ideals
($\mathrm{Irc}(R)$), principal ideals
($\mathrm{Prn}(R)$),  regular ideals
($\mathrm{Reg}(R)$),  finitely generated ideals
($\mathrm{Fgn}(R)$), strongly  irreducible ideals ($\mathrm{Irs}(R)$). We use the symbol $\mathrm{X}(R)$ to denote an arbitrary spectrum of a ring $R$, and for all spectra, we assume $R\notin \mathrm{X}(R).$    By a \emph{sub-spectrum} of a spectrum $\mathrm{X}(R)$, we mean a spectrum that is also a subset of $\mathrm{X}(R)$. 

\begin{figure}[ht!]  
\begin{center}  
$    
\xymatrix@R=15pt@C=12pt{
&*+[F-:<3pt>]\txt{\footnotesize $\mathrm{X}(R)$}\ar@{-}[dl]\ar@{-}[dr]\\*+[F-:<3pt>]\txt{\footnotesize $\mathrm{Irs}(R)$\\\footnotesize[\eqref{mic}  holds]}\ar@{-}[d] &&*+[F-:<3pt>]\txt{\footnotesize$\mathrm{Prp}(R)$\\\footnotesize[\eqref{mic} does not hold]}\ar@{-}[d] \\ 
*+[F-:<3pt>]\txt{\footnotesize$\mathrm{Spec}(R)$, $\mathrm{Spn}(R)$, $\mathrm{Max}(R)$, \textit{etc}.}&&*+[F-:<3pt>]\txt{\footnotesize $\mathrm{Rad}(R)$, $\mathrm{Prn}(R)$, $\mathrm{Min}(R)$,\\ \footnotesize $\mathrm{Fgn}(R)$, $\mathrm{Irr}(R)$, $\mathrm{Prm}(R)$,  \textit{etc}.}  
 }
$  
\end{center} 
\caption{Classification of  $\mathrm{X}(R)$}
\label{fig:cxr}
\end{figure}
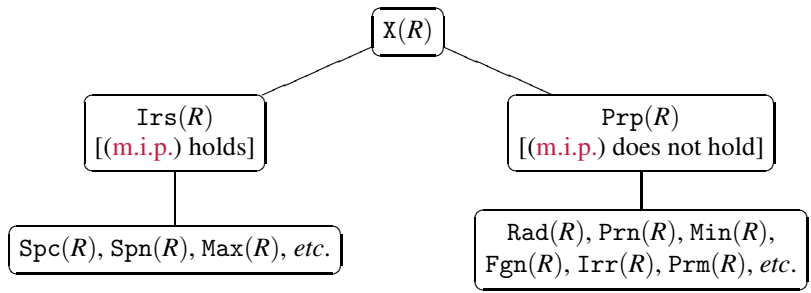  

Now $\mathrm{Irs}(R)$ is the `largest' spectrum on which we can define a hull-kernel topology and hence study structure spaces. The other examples are three of its sub-spectra (see Figure \ref{fig:cxr}). But that is no longer true for $\mathrm{Prp}(R)$ or for its sub-spectra (see Figure  \ref{fig:cxr}) simply  because for them  (\ref{mic}) does not hold (see Examples \ref{nkc1}, \ref{nkc2}, and  \ref{nkc3}). The way out for them is to generate a  ``hull-kernel-type''  topology using a closed subbase (see  Definition
\ref{sbt}). This closed-subbase topology is also known as a coarse lower topology (see \cite[A.8, p.\,589]{DST19}) or a lower topology (see \cite[Definition O-5.4, p.\,43]{G et. al.}). The motive for considering this topology is twofold. On one hand, hull and kernel maps are also involved in the construction of the closed-subbase topology as they do for hull-kernel topology; on the other hand, this closed-subbase topology coincides with  the hull-kernel topology when a spectrum $\mathrm{X}(R)$ satisfies (\ref{mic}). We call a spectrum $\mathrm{X}(R)$ endowed with the above mentioned closed-subbase topology an ideal space (see Definition \ref{sbt}). Therefore, we may see an ideal space as an extension of the notion of a structure space.    
 
To have a unified approach to study ideal spaces of all spectra, throughout the paper, we use subbasic closed sets in  statements and proofs. Whenever a statement or a proof of a result of an ideal space differs from that of a structure space, we highlight them.   
It is evident that many of the results of this paper can be extended to more general structures (\textit{e.g.} multiplicative lattices)  than rings. However, for simplicity of exposition and to be accessible to a wider audience, we have chosen to retain the ring terminology.   

Let us note the conventions of terminology in this paper. All rings considered in this paper are assumed to be commutative and to possess an identity element; all the ring homomorphisms  are  assumed to map identity element to identity  element. The notation and the terminology are largely the standard ones. Thus, for instance, throughout this paper, $R$ will denote a ring. By a \emph{non-trivial} idempotent element in a ring we mean an  idempotent element not equal to $0$ or $1.$ We denote the set of all ideals of a ring $R$ by $\mathrm{Idl}(R)$ and by the notation $\mathfrak{o},$ we denote the zero ideal of $R$; whereas  we denote the radical of an ideal $\mathfrak{a}$ by $\sqrt{\mathfrak{a}}.$ By a semi-simple ring $R$, we mean the Jacobson radical of $R$ is the zero ideal. The symbols $\emptyset$ and $\mathcal{P}\!(S)$ respectively denote the empty set and the powerset of $S$.   
An ideal $\mathfrak{a}$ is called \emph{proper} if $\mathfrak{a}$ is not equal to $R$. If $S$ is a subset of $R,$ the ideal generated by $S$ is denoted by $\langle S\rangle.$ 

We describe the contents of the paper briefly.
Section \ref{sec2} is dedicated to study the properties of hull and kernel maps. 
Section \ref{basb} is devoted to structure spaces, where in Theorem \ref{hkt}, we see the characterisation of spectra for structure spaces. Moreover, Theorem \ref{kclcb} characterises the ``hull-kernel closure operator'' in terms of a closed base. Section \ref{idsp} starts with examples of spectra on which we cannot define hull-kernel topologies, and hence motivates for the closed-subbase topology on them. In the following three subsections we study various topological properties of ideal spaces. Canonical continuous maps and homeomorphisms between ideal spaces are studied in \ref{conti}, whereas \ref{sepax}, is about compactness and separation axioms. In particular, we characterise (Theorem \ref{sob}) ideal spaces that are sober.    
Finally, in \ref{comconn}, we discuss connectedness. We introduce (Definition \ref{sconn}) the notion of a strongly connected space and show its relation with  connectedness. We also give a sufficient condition for a spectrum to be connected (Theorem \ref{conis}).  

\section{Hull and Kernel maps}\label{sec2}

Before we discuss about structure spaces (see Section \ref{basb}) and extend them to ideal spaces (see Section \ref{idsp}), we first study properties of those two maps that induce the corresponding topologies. 
 
\begin{definition}\label{hkm} 
 A \emph{hull} map $\mathrm{h}\colon \mathrm{Idl}(R)\to \mathcal{P}\!\!\left(\mathrm{X}(R)\right)$ is defined by $\mathrm{h}(\mathfrak{a})=\{\mathfrak{s}\in \mathrm{X}(R)\mid \mathfrak{a} \subseteq \mathfrak{s}\}$. A \emph{kernel} map $\mathrm{k}\colon \mathcal{P}\!\!\left(\mathrm{X}(R)\right)\to \mathrm{Idl}(R)$ is defined by
$\mathrm{k}(S)=\cap \,S=\cap\{\mathfrak{a}\mid \mathfrak{a}\in S \}.$ 
\end{definition}

In what follows, and throughout the paper, we view $\mathrm{Idl}(R)$ and all its subsets as a poset, partially ordered by inclusion.
Similarly, we view $\mathcal{P}\!\!\left(\mathrm{X}(R)\right)$ as a poset (of course, a Boolean algebra) partially ordered by inclusion. Note then that the hull and kernel maps are order-reversing maps. So the Galois connection referred to in Proposition \ref{muls} is to be understood in this context, so that it is what is usually referred to as an antitone Galois connection.

\begin{proposition}\label{muls}

 For any spectrum $\mathrm{X}(R),$ the maps $\mathrm{h}$ and $\mathrm{k}$ have the following properties.   
 
\begin{enumerate}[\upshape (i)]

\item\label{gcac} The pair $(\mathrm{h}, \mathrm{k})$ of maps forms a Galois connection and $\mathrm{hk}$ is an algebraic closure operation on $\mathrm{X}(R).$  
  
\item \label{hkd} The maps $\mathrm{h}$ and $\mathrm{k}$ are order-reversing.  Moreover, $\mathrm{h}(R)=\emptyset,$ $\mathrm{h}(\mathfrak{o})=\mathrm{X}(R),$ and $\mathrm{k}(\emptyset)=R$. 
 
\item \label{ihs} For $\mathfrak{a},$ $\mathfrak{b}\in  \mathrm{Idl}(R)$, $\mathrm{h}(\mathfrak{a})\cup \mathrm{h}(\mathfrak{b}) \subseteq \mathrm{h}(\mathfrak{a}\cap \mathfrak{b})\subseteq \mathrm{h}(\mathfrak{ab}).$

\item\label{insum} For a family $\{ \mathrm{h}(\mathfrak{a}_{\alpha}) \}_{\alpha \in \Lambda},$ $\bigcap_{\alpha \in \Lambda}\mathrm{h}(\mathfrak{a}_{\alpha})=\mathrm{h}\left( \sum_{\alpha \in \Lambda}\mathfrak{a}_{\alpha} \right)$.   

\item\label{unin} For a family $\{ S_{ \lambda} \}_{\lambda \in \Gamma},$  $\mathrm{k}\left( \bigcup_{\lambda \in \Gamma}S_{ \lambda} \right)=\bigcap_{\lambda \in \Gamma} \mathrm{k}\left(S_{ \lambda}\right)$. 
\item\label{arad} For $\mathfrak{a}\in \mathrm{Idl}(R),$ $\mathrm{h}(\mathfrak{a})\supseteq\mathrm{h}(\sqrt{\mathfrak{a}}).$
\end{enumerate}
\end{proposition}  
    
\begin{proof}  
The proofs of (\ref{hkd})--(\ref{arad}) are straightforward. We prove only (\ref{gcac}). To show that the pair $(\mathrm{h}, \mathrm{k})$ forms a Galois connection, it is sufficient to show $S \subseteq \mathrm{h}(\mathfrak{a})$ if and only if $\mathrm{k}(S)\supseteq \mathfrak{a},$ for all $S\in \mathcal{P}\!(\mathrm{X}(R))$ and for all $\mathfrak{a}\in \mathrm{Idl}(R).$ Suppose $S \subseteq \mathrm{h}(\mathfrak{a})$ and $x\in \mathfrak{a}.$  Then $x\in \mathfrak{a}\subseteq \mathfrak{s}$ for all $\mathfrak{s}\in S,$ and that implies $x\in \cap
S=\mathrm{k}(S).$ Conversely, $\mathrm{k}(S)\supseteq \mathfrak{a}$ implies $\mathfrak{a}\subseteq \mathfrak{s}$ for all $\mathfrak{s} \in S,$ and that gives $S\subseteq \mathrm{h}(\mathfrak{a}).$ 
Since every Galois connection induces an algebraic closure operation, the same holds  for $\mathrm{hk}.$ 
\qedhere
\end{proof}   

We note that (\ref{insum}) and (\ref{unin}) could also be proved via the Galois connection established in (\ref{gcac}), as follows. Consider $\mathrm{h}$ as  a mapping $\mathrm{h}\colon \mathrm{Idl}(R)\to\mathcal P(\mathrm{X}(R))^{\rm op}$ and $\mathrm{k}$ as a mapping  $\mathrm{k}\colon \mathcal{P}(\mathrm{X}(R))^{\rm op}\to \mathrm{Idl}(R)$. Then $\mathrm h$ and $\mathrm{k}$ are order-preserving, and so, from the Galois connection, $\mathrm{h}$ preserves joins and $\mathrm{k}$ preserves meets. In $\mathrm{Idl}(R)$, joins  are sums  and meets are intersections. In $\mathcal{P}(\mathrm{X}(R))^{\rm op}$, joins are intersections and meets are unions. The results follow from from this. 
There are equalities in (\ref{ihs}) and in (\ref{arad}) for $\mathrm{Irs}(R)$ and  $\mathrm{Spec}(R)$ respectively. The following result characterises spectra for which equality holds in (\ref{arad}). 

\begin{proposition}\label{hrrad}
 The following are equivalent.
\begin{enumerate}[\upshape (i)]
\item \label{hahrada}
$\mathrm{h}(\mathfrak{a})=\mathrm{h}(\sqrt{\mathfrak{a}})$ for every $\mathfrak{a}\in \mathrm{Idl}(R)$.
\item
$\mathrm{h}(\mathfrak{a})=\mathrm{h}(\sqrt{\mathfrak{a}})$ for every $\mathfrak{a}\in \mathrm{X}(R)$.  
\item 
Every ideal in $\mathrm{X}(R)$ is a radical ideal. 
\end{enumerate}
\end{proposition}
 
\begin{proof}
(i)$\Rightarrow$(ii). This is immediate because $\mathrm{X}(R)\subseteq\mathrm{Idl}(R)$.
 
(ii)$\Rightarrow$(iii). Let $\mathfrak{a}\in\mathrm{X}(R)$. Then $\mathfrak{a}\in\mathrm{h}(\mathfrak{a})$, and so, by (ii), $\mathfrak{a}\in\mathrm{h}(\sqrt{\mathfrak{a}})$. But this implies $\sqrt{\mathfrak{a}}\subseteq \mathfrak{a}$, whence we have $\mathfrak{a}=\sqrt{\mathfrak{a}}$, showing that $\mathfrak{a}$ is a radical ideal. 

(iii)$\Rightarrow$(i). Let $\mathfrak{a}\in\mathrm{Idl}(R)$. If $\mathrm{h}(\mathfrak{a})=\emptyset$, then $\mathrm{h}(\sqrt{\mathfrak{a}})=\emptyset$ because $\mathrm{h}(\sqrt{\mathfrak{a}})\subseteq \mathrm{h}(\mathfrak{a})$, hence 	$\mathrm{h}(\mathfrak{a})=\mathrm{h}(\sqrt{\mathfrak{a}})$. If $\mathrm{h}(\mathfrak{a})\neq \emptyset$, consider  any $\mathfrak{b}\in\mathrm{h}(\mathfrak{a})$. Then $\mathfrak{b}\in\mathrm{X}(R)$, and is therefore a radical ideal, by (iii). Furthermore, $\mathfrak{a}\subseteq \mathfrak{b}$, which then implies that $\sqrt{\mathfrak{a}}\subseteq \sqrt{\mathfrak{b}}=\mathfrak{b}$, whence $\mathfrak{b}\in\mathrm{h}(\sqrt{\mathfrak{a}})$. Therefore, $\mathrm{h}(\mathfrak{a})\subseteq \mathrm{h}(\sqrt{\mathfrak{a}})$, and hence $\mathrm{h}(\mathfrak{a})=\mathrm{h}(\sqrt{\mathfrak{a}})$, as required.  
\end{proof}  

From Proposition \ref{hrrad}, it is clear that for all $\mathfrak{a}\in \mathrm{Idl}(R)$ the equality $\mathrm{h}(\mathfrak{a})=\mathrm{h}(\sqrt{\mathfrak{a}})$ holds only for the spectrum $\mathrm{Rad}(R)$ or its sub-spectra. It is natural to ask whether the notion of radical $\sqrt{\mathfrak{a}}$ of an ideal $\mathfrak{a}$ can be generalised in such a fashion that a similar identity as in Proposition \ref{hrrad}(\ref{hahrada}) still holds for all spectra. The following Definition \ref{grad} is an answer of that question. Moreover, it helps to derive a number of properties (see Proposition \ref{hrx}) for an arbitrary spectrum $\mathrm{X}(R)$ that specifically $\mathrm{Spec}(R)$ has.   

\begin{definition} \label{grad} 
 For a spectrum $\mathrm{X}(R)$ and for $\mathfrak{a}\in\mathrm{Idl}(R),$ we define the ideal $\sqrt[\mathrm{X}]{\mathfrak{a}}=\mathrm{kh}(\mathfrak{a}).$
\end{definition} 

\begin{proposition}\label{hrx} 
Suppose $\mathrm{X}(R)$ is
a spectrum.
\begin{enumerate}[\upshape (i)]
\label{rxlr}  
\item\label{axaa} For every $\mathfrak{a}\in \mathrm{Idl}(R),$ $\mathfrak{a}\subseteq \sqrt[\mathrm{X}]{\mathfrak{a}}.$

\item \label{axa}   
If $\mathfrak{a}\in\mathrm{X}(R),$ then $\mathfrak{a}=\sqrt[\mathrm{X}]{\mathfrak{a}}$.

\item \label{hahxa} For all $\mathfrak{a} \in \mathrm{Idl}(R),$ $\mathrm{h}(\mathfrak{a})=\mathrm{h}(\sqrt[\mathrm{X}]{\mathfrak{a}}).$

\item\label{xbxa} A necessary and sufficient condition for the inclusion $\mathrm{h}(\mathfrak{a})\subseteq \mathrm{h}(\mathfrak{b})$ to hold is  $\sqrt[\mathrm{X}]{\mathfrak{b}}\subseteq\sqrt[\mathrm{X}]{\mathfrak{a}}$.

\item\label{regc} Let $R$ be regular ring and $\mathfrak{a},$ $\mathfrak{b}\in \mathrm{Idl}(R)$. A  necessary and sufficient condition for the inclusion $\mathrm{h}(\mathfrak{a})\subseteq \mathrm{h}(\mathfrak{b})$ to hold is  $\mathfrak{b}\subseteq\mathfrak{a}$.

\item\label{chchk} The collections $\mathcal{C}_{\mathrm{h}}=\{ \mathrm{h}(\mathfrak{a})\mid \mathfrak{a}\in \mathrm{Idl}(R)\}$ and  $\mathcal{C}_{\mathrm{hk}}=\{ \mathrm{hk}(S)\mid S\in \mathcal{P}(\mathrm{X}(R))\}$ of sets are identical.   
\end{enumerate}
\end{proposition}

\begin{proof}
(i). The first inclusion follows from Definition \ref{grad}.

(ii). If $\mathfrak{a}\in\mathrm{X}(R),$ then $\sqrt[\mathrm{X}]{\mathfrak{a}}={\cap}\{\mathfrak{s}\in\mathrm{X}(R)\mid \mathfrak{a}\subseteq \mathfrak{s}\}=\mathfrak{a}.$

(iii). By (\ref{axaa}), $\mathfrak{a}\subseteq \sqrt[\mathrm{X}]{\mathfrak{a}},$ and hence we have $\mathrm{h}(\mathfrak{a})\supseteq\mathrm{h}(\sqrt[\mathrm{X}]{\mathfrak{a}}).$ Now the other inclusion follows from Definition \ref{grad}.
 
(iv). \textit{Necessity}. Follows from order-reversing property of $\mathrm{k}.$ 

\textit{Sufficiency}. Now  suppose $\sqrt[ \mathrm{X}]{\mathfrak{b}}\subseteq \sqrt[\mathrm{X}]{\mathfrak{a}}$, and  $\mathfrak{s}\in \mathrm{h}(\mathfrak{a})$. Then $\mathfrak{s}\in \mathrm{X}(R)$ and $\mathfrak{a}\subseteq \mathfrak{s}$, so that $\sqrt[\mathrm{X}]{\mathfrak{a}}\subseteq \mathfrak{s}$. Therefore 
$
\mathfrak{s}\in\mathrm{h}(\mathfrak{s})\subseteq \mathrm{h}(\sqrt[\mathrm{X}]{\mathfrak{a}})\subseteq \mathrm{h}(\sqrt[\mathrm{X}]{\mathfrak{b}})\subseteq\mathrm{h}(\mathfrak{b}),
$
showing that $\mathrm{h}(\mathfrak{a})\subseteq \mathrm{h}(\mathfrak{b})$.

(v). \textit{Necessary}. Since $R$ is regular, by \cite[Theorem 49, p.\,148]{M48}, $\mathrm{Idl}(R)=\mathrm{Rad}(R),$ and hence $\mathrm{h}(\mathfrak{a})\subseteq \mathrm{h}(\mathfrak{b})$ implies $\mathfrak{b}\subseteq \cap \mathrm{h}(\mathfrak{a})=\sqrt[\mathrm{X}]{\mathfrak{a}}=\mathfrak{a}.$
 
\textit{Sufficiency}. Follows from order-reversing property of $\mathrm{h}.$ 
      
(vi). If $\mathfrak{a}\in \mathrm{Idl}(R),$ then we have $\mathrm{h}(\mathfrak{a})=\mathrm{h}(\sqrt[\mathrm{X}]{\mathfrak{a}})=\mathrm{h}(\mathrm{kh}(\mathfrak{a}))=\mathrm{hk}(S),$ where $S=\{\mathfrak{s}\in\mathrm{X}(R)\mid \mathfrak{a}\subseteq \mathfrak{s}\}$. On the other hand, if $S\in \mathcal{P}(\mathrm{X}(R)),$ then $\mathrm{hk}(S)=\mathrm{h}(\cap S)=\mathrm{h}(\mathfrak{b}),$ where $\mathfrak{b}=\cap S\in \mathrm{Idl}(R).$
\end{proof}   
  
\section{Structure spaces}\label{basb}  
 
Before we extend the notion of a structure space to include all spectra, we first see how hull-kernel topology (see Definition \ref{dhk}) on a spectrum $\mathrm{X}(R)$ is related to the fact that $\mathrm{X}(R)$ satisfies (\ref{mic}), and hence we will find for which spectra we can define hull-kernel topologies. Theorem \ref{hkt} classifies such spectra, and for a proof, we refer to \cite[Section 2.2, p.\,11]{M53}.      
 
\begin{theorem}\label{hkt}
 For the composite $\mathrm{hk}$ to be a Kuratowski closure operation on a spectrum $\mathrm{X}(R),$  it is necessary and sufficient that $\mathrm{X}(R)$ has the property    
\[
\mathfrak{a}\cap \mathfrak{b}\subseteq \mathfrak{s}\quad \text{implies}\quad  \mathfrak{a}\subseteq \mathfrak{s}\;\; \text{or}\;\; \mathfrak{b}\subseteq \mathfrak{s} \label{mic} \tag{$\text{m.i.p.}$}  
\] for all $\mathfrak{a},$ $\mathfrak{b}\in \mathrm{im}(\mathrm{k})$ and for all $\mathfrak{s}\in  \mathrm{X}(R).$  
\end{theorem}  
 

  
\begin{definition}\label{dhk}  

 The induced topology on $\mathrm{X}(R)$ by the Kuratowski closure operation $\mathrm{hk}$ is called the  \emph{hull-kernel} topology or \emph{stone} topology; $\cap
S$ is called the \emph{kernel} of $S,$ and $\mathrm{h}(\cap S)$ is called the \emph{hull} of $\cap S$. A spectrum $\mathrm{X}(R)$ endowed with the hull-kernel topology is called a \emph{structure space}. 
\end{definition}

Although for $\mathrm{Spec}(R)$, the property (\ref{mic}) holds  for all $\mathfrak{a},$ $\mathfrak{b}\in \mathrm{Idl}(R),$ restricting it to $\mathrm{im}(\mathrm{k})$ is sufficient for $\mathrm{hk}$ to be a Kuratowski closure operation on any spectrum $\mathrm{X}(R)$. The following proposition characterises $\mathrm{X}(R)$ for which we have $\mathrm{im}(\mathrm{k})=\mathrm{X}(R)$. 

\begin{proposition}
 For a  spectrum $\mathrm{X}(R)$ to be equal to $\mathrm{im}(\mathrm{k})$, it is necessary and sufficient that $\mathrm{X}(R)$ is closed under intersections. 
\end{proposition}

\begin{proof}
\textit{Necessity}. 
Assume that $\mathrm{X}(R)=\mathrm{im}(\mathrm{k})$, and let $\{\mathfrak{a}_i\}$ be a collection of ideals belonging to $\mathrm{X}(R)$. For each $\alpha$ pick  $S_\alpha\in \mathcal P(\mathrm{X}(R))$ such that $\mathfrak{a}_\alpha=\mathrm{k}(S_\alpha)$.  Therefore
\[
{\bigcap_\alpha}\mathfrak{a}_\alpha={\bigcap_\alpha}\mathrm{k}(S_\alpha)=\mathrm{k}\Big({\bigcup_\alpha}S_\alpha\Big)\in\mathrm{im}(\mathrm{k})=\mathrm{X}(R),
\]
which shows that $\mathrm{X}(R)$ is closed under intersections.

\textit{Sufficiency}. Assume that $\mathrm{X}(R)$ is closed under intersections. If $\mathfrak a\in\mathrm{X}(R)$, then, for the  element $\{\mathfrak a\}$ of $\mathcal P(\mathrm{X}(R))$ we have $\mathrm{k}(\{\mathfrak a\})= \mathfrak a$, and that implies $\mathrm{im}(\mathrm{k})\supseteq \mathrm{X}(R).$ To get the claimed equality, what left is to show that $\mathrm{im}(\mathrm{k})\subseteq \mathrm{X}(R)$. Let $\mathfrak{a}\in \mathrm{im}(\mathrm{k})$, and pick $S\in\mathcal P(\mathrm{X}(R))$ such that $\mathfrak{a}=\mathrm{k}(S)$. Note that $S\subseteq \mathrm{X}(R)$, and hence ${\cap}S\in \mathrm{X}(R)$, by the hypothesis. But  ${\cap}S=\mathrm{k}(S)=\mathfrak{a}$; hence $\mathfrak{a}\in \mathrm{X}(R)$, which establishes the desired containment.   
\end{proof} 

The spectrum that   satisfies \eqref{mic} for all $\mathfrak{a},$ $\mathfrak{b}\in \mathrm{Idl}(R),$ is  nothing but the spectrum of strongly irreducible ideals $\mathrm{Irs}(R)$ introduced in \cite{F49} for commutative rings, under the name  \emph{primitive} ideals. In \cite[p.\,301, Exercise 34]{B72}, the ideals of the same spectrum are called \emph{quasi-prime} ideals. The terminology ``strongly irreducible'' was first used for noncommutative rings by Blair in \cite{B53}. Since $\mathrm{Spec}(R),$ $\mathrm{Spn}(R),$ and $\mathrm{Max}(R)$ are sub-spectra of $\mathrm{Irs}(R),$ they also satisfy \eqref{mic} for all $\mathfrak{a},$ $\mathfrak{b}\in \mathrm{Idl}(R),$ and hence these are the examples of spectra on which we can define hull-kernel topology. Moreover, an ideal $\mathfrak{p}$ to be prime, it is necessary and sufficient that $\mathfrak{p}$ is strongly irreducible and radical, (see \cite[Theorem 2.1]{A08}).    
It is worth it to mention that the property (\ref{mic}) is equivalent to the condition $\mathrm{h}(\mathfrak{a})\cup \mathrm{h}(\mathfrak{b})=\mathrm{h}(\mathfrak{a}\cap \mathfrak{b}),$ where $\mathfrak{a},$ $\mathfrak{b}\in \mathrm{im}(\mathrm{k}),$ \textit{cf.} Proposition \ref{muls}(\ref{ihs}). 

Whenever $\mathrm{hk}$ is a Kuratowski closure operation on an $\mathrm{X}(R)$, the collection $\mathcal{C}_{\mathrm{hk}}$ of subsets of $\mathrm{X}(R))$ are the closed subsets of the hull-kernel  topology on $\mathrm{X}(R)$ and hence, $\mathcal{C}_{\mathrm{hk}}$ is a closed base. The following theorem says that $\mathcal{C}_{\mathrm{hk}}$ characterises $\mathrm{hk}$ as Kuratowski closure operation. 
 
\begin{theorem}\label{kclcb} 
 For $\mathrm{hk}$ to be a Kuratowski closure operation on a spectrum $\mathrm{X}(R),$ it is necessary and sufficient that $\mathcal{C}_{\mathrm{hk}}$ is a closed base of the hull-kernel topology on $\mathrm{X}(R).$ 
\end{theorem}

\begin{proof}      
\textit{Necessity}. Let $\mathrm{hk}$ be a Kuratowski closure operation on $\mathrm{X}(R).$ Then $\mathcal{C}_{\mathrm{hk}}$ is the set of closed subsets of $\mathrm{X}(R),$ and hence $\mathcal{C}_{\mathrm{hk}}$ is also a closed base for $\mathrm{X}(R).$
   
\textit{Sufficiency}.
In the proof, we shall use the fact that if $\mathcal{B}$ is a closed base on a space $X$, then the  closure operator induced by the topology  is expressible in terms of $\mathcal{B}$, in the sense that, for any $T\subseteq  X$,  
$
\mathrm{cl}(T)={\cap}\{B\in \mathcal{B}\mid T\subseteq B\}.
$
Suppose that $\mathcal{C}_{\mathrm{hk}}$ is a closed base for $\mathrm{X}(R)$. Let $\mathrm{cl}(-)$ be the resulting closure operation. We shall be done if we can show that 	$\mathrm h\mathrm k(-)=\mathrm{cl}(-)$. Consider, then, any $T\subseteq \mathrm{X}(R)$. As observed above,
\[ 
\mathrm{cl}(T)=\cap\{\mathrm{hk}(S)\in \mathcal{C}_{\mathrm{hk}}\mid  T\subseteq \mathrm{hk}(S)\}.
\]
Since $\mathrm{hk}(T)\in \mathcal{C}_{\mathrm{hk}}$ and $T\subseteq \mathrm{hk}(T)$, we deduce from  the equality $\mathrm{hk}(\mathrm{hk}(T))=\mathrm{hk}(T)$ that $\mathrm{cl}(T)\subseteq \mathrm{hk}(T)$. To have the reverse inclusion, let $\mathfrak a\in \mathrm{hk}(T)$. Consider any $S\in \mathcal{P}(\mathrm{X}(R))$ with $T\subseteq \mathrm{hk}(S)$. Then, using the facts that $\mathrm{h}$ and $\mathrm k$ are in a Galois connection, and $\mathrm{k}$ is order-reversing, we have 
\[
\mathrm{k}(S)=\mathrm{khk}(S)\subseteq \mathrm{k}(T).
\]
Since  $\mathrm{k}(T)\subseteq \mathfrak a$, as $\mathfrak{a}\in \mathrm{hk}(T)$, it follows that $\mathrm{k}(S)\subseteq \mathfrak{a}$. Thus, $\mathfrak{a}\in\mathrm{hk}(S)$.  In all then, we have shown that $\mathfrak{a}$ belongs to every set included in the intersection that describes $\mathrm{cl}(T)$, and so $\mathfrak{a}\in \mathrm{cl}(T)$, and so $\mathrm{hk}(T)\subseteq \mathrm{cl}(T)$. Consequently, $\mathrm{hk}(T)=\mathrm{cl}(T)$, and since $T$ is an arbitrary subset of $\mathrm{X}(R)$, it follows that $\mathrm{hk}(-)=\mathrm{cl}(-)$, and hence $\mathrm{hk}$ is a Kuratowski closure operation.	
\end{proof}
 
For $\mathrm{Spec}(R),$ instead of $\mathcal{C}_{\mathrm{hk}}$, we  use $\mathcal{C}_{\mathrm{h}}$ (thanks to the equality in Proposition \ref{hrx}(\ref{chchk}))  as the collection of closed     subsets of a Zariski (= hull-kernel) topology on $\mathrm{Spec}(R).$ Moreover, $\mathrm{h}(\mathfrak{a})=\bigcap_{a\in \mathfrak{a}}\mathrm{h}(\langle \{a \}\rangle)$ implies that the collection $\{\mathrm{h}(\mathfrak{a})\mid \mathfrak{a}\in \mathrm{Prn}(R)\}$ is a closed base of a Zariski  topology on $\mathrm{Spec}(R).$  
  
\section{Ideal spaces}\label{idsp} 

We have seen that $\mathrm{hk}$ is a Kuratowski closure operation on $\mathrm{Irs}(R)$ or on its sub-spectra, and hence we can define hull-kernel topologies on them. Let us now see examples of spectra that fail to satisfy (\ref{mic}), and hence $\mathrm{hk}$ does not induce hull-kernel topologies on them.  

\begin{example}\label{nkc1}
 For $\mathrm{Min}(R)$, consider the ring $R = \mathds{Z}_2\times  \mathds{Z}_2 \times  \mathds{Z}_2$,                   and the minimal ideals $\mathfrak{x}$, $\mathfrak{y}$, and $\mathfrak{z}$, generated by $(1,0,0),$ $(0,1,0),$ and $(0,0,1)$ respectively. Although $\mathfrak{x}\cap \mathfrak{y}\subseteq \mathfrak{z},$ but $\mathfrak{x}\nsubseteq \mathfrak{z}$ and $\mathfrak{y}\nsubseteq \mathfrak{z}.$  
\end{example} 
 
\begin{examples}\label{nkc2} 
 For $\mathrm{Prp}(R),$ $\mathrm{hk}$ is not a Kuratowski closure operation. To see this, let $R=\mathds{Z}$ and consider $S=\{2\mathds{Z}\},$ $S'=\{3\mathds{Z}\},$ and $\mathfrak{a}=6\mathds{Z}.$ Since $\mathrm{k}(\{2\mathds{Z}\})\cap \mathrm{k}(\{3\mathds{Z}\})=2\mathds{Z}\cap 3\mathds{Z}=6\mathds{Z}\subseteq 6\mathds{Z},$ but $2\mathds{Z}\nsubseteq 6\mathds{Z}$ and $3\mathds{Z}\nsubseteq 6\mathds{Z}.$ The same argument also holds for $\mathrm{Prn}(R),$ and hence for $\mathrm{Fgn}(R).$  Since
$6\mathds{Z}\in \mathrm{Rad}(R),$ by the same counterexample, $\mathrm{hk}$ is also not a Kuratowski closure operation on $\mathrm{Rad}(R).$ 
\end{examples} 

\begin{examples}\label{nkc3}
 Since $\mathrm{Irr}(R)$ and $\mathrm{Prm}(R)$ are in general not sub-spectra of $\mathrm{Irs}(R)$ (see \cite{S16}  and \cite{HRR02}), these spectra also do not satisfy (\ref{mic}).
\end{examples} 

If $\mathrm{hk}$ is not a Kuratowski closure  operation on a spectrum $\mathrm{X}(R)$, then by Theorem \ref{kclcb}, the collection   $\mathcal{C}_{\mathrm{hk}}$ (or, equivalently $\mathcal{C}_{\mathrm{h}}$) \emph{also} does not form a closed base of a  hull-kernel topology. Nevertheless, for such a  spectrum, we still have $\bigcap_{ S\,\in\, \mathcal{P}\!(\mathrm{X}(R))}\mathrm{hk}(S) =\emptyset$ (or, equivalently $\bigcap_{ \mathfrak{a}\,\in\, \mathrm{Idl}(R))}\mathrm{h}(\mathfrak{a}) =\emptyset$), which follows from Proposition \ref{muls}(\ref{hkd})). Hence, by \cite[Theorem\,15\,A.13., p.\,254]{C66}, the collection $\mathcal{C}_{\mathrm{hk}}$ (or, equivalently $\mathcal{C}_{\mathrm{h}}$) as  a closed subbase, generates a unique topology (denoted by $\tau_{\scriptscriptstyle\mathrm{X}(R)}$)  on $\mathrm{X}(R)$. The topology $\tau_{\scriptscriptstyle\mathrm{X}(R)}$ is called the coarse lower topology or the lower topology. A hull-kernel topology coincides with $\tau_{\scriptscriptstyle\mathrm{X}(R)}$ when $\mathrm{X}(R)$ satisfies (\ref{mic}). Note that when we endow a lower topology (generated by $\mathcal{C}_{\mathrm{h}}$) on an $\mathrm{X}(R),$ it coincides with subspace topology of the lower topology on $\mathrm{Idl}(R)$.   We are now ready to extend the definition of a structure space of a ring.   
 
\begin{definition}\label{sbt} 
 A spectrum $\mathrm{X}(R)$ endowed with the $\tau_{\scriptscriptstyle\mathrm{X}(R)}$ topology is called an \emph{ideal space}.
\end{definition}  

\begin{notation}
 Whenever the collection  $\mathcal{C}_{\mathrm{h}}$  forms only a closed subbase of an ideal space $\mathrm{X}(R),$ we denote the subbase by $\mathcal{S}_{\mathrm{h}}$, and whenever the collection $\mathcal{C}_{\mathrm{h}}$ also forms a closed base,  we denote the corresponding base by $\mathcal{B}_{\mathrm{h}}$. 
\end{notation} 

In the next three subsections, we will study various topological properties of ideal spaces. A number of these properties are direct extension from that of the structure spaces of $\mathrm{Spec}(R)$ or $\mathrm{Max}(R)$, whereas for others to hold for an ideal space requires additional assumptions. 

\subsection{Maps between ideal spaces} \label{conti}   
 
Although the inverse image of a prime ideal under a ring homomorphism is a prime ideal, the same is not true for an arbitrary spectrum $\mathrm{X}(R).$ For example, it fails to hold in $\mathrm{Max}(R).$ Therefore, to have a  continuous map between two  ideal spaces of same spectra (as we obtain for $\mathrm{Spec}(R)$), we need the following property.  

\begin{definition}
 We say a spectrum $\mathrm{X}(R)$ satisfies the \emph{contraction-ideal} property if for any ring homomorphism $f\colon R\to R',$ the inverse image  $f\inv(\mathfrak{b})$ is in $\mathrm{X}(R)$, whenever $\mathfrak{b}$ is in $\mathrm{X}(R').$ 
\end{definition}

With this assumption on spectra and using subbasic-closed-set formulation of continuity, we obtain the following properties. The arguments of these proofs are more or less identical as we see for $\mathrm{Spec}(R),$ except the fact that we have to use subbasic closed sets. 
    
\begin{proposition}\label{conmap}
 Let $\mathrm{X}(R)$ be a spectrum satisfying the contraction-ideal property. Let $f\colon R\to R'$ be a ring homomorphism  and $\mathfrak{b}$ be in $\mathrm{X}(R').$ 
\begin{enumerate}[\upshape (i)] 
\item \label{contxr} The map $f^*\colon  \mathrm{X}(R')\to \mathrm{X}(R)$ defined by  $f^*(\mathfrak{b})=f\inv(\mathfrak{b})$ is    continuous.

\item If $f$ is  surjective, then the ideal space $\mathrm{X}(R')$ is homeomorphic to the closed subset $\mathrm{h}(\mathrm{ker}(f))$ of the ideal space $\mathrm{X}(R).$

\item For  $f^*(\mathrm{X}(R'))$ to be dense in $\mathrm{X}(R),$ it is necessary and sufficient that $\mathrm{ker}(f)\subseteq \bigcap_{\mathfrak{s}\in \mathrm{X}(R)}\mathfrak{s}.$ 
  
\item If  $S$ is a multiplicative closed subset of a ring $R$ and if $R_{ S}$ is the localisation of $R$ at $S$, then there is a closed, continuous, and injective map from the ideal space  $\mathrm{X}(R_{ S})$ to the ideal space $\mathrm{X}(R)_S=\{ \mathfrak{s}\in \mathrm{X}(R)\mid \mathfrak{s}\cap S=\emptyset\}$.  
\end{enumerate}
\end{proposition}
  
\begin{proof}      
(i). Let $\mathrm{h}(\mathfrak{a})$ be a   subbasic closed set of the ideal  space $\mathrm{X}(R).$ Observe  that $$(f^*)\inv(\mathrm{h}(\mathfrak{a})) =\{ \mathfrak{b}\in  \mathrm{X}(R')\mid f\inv(\mathfrak{b})\in \mathrm{h}(\mathfrak{a})\}=\{\mathfrak{b}\in \mathrm{X}(R')\mid f(\mathfrak{a})\subseteq \mathfrak{b}\}=\mathrm{h}(\langle f(\mathfrak{a})\rangle),$$ and hence the map $f^*$  continuous.     
  
(ii). Since  $\mathfrak{o}\subseteq \mathfrak{b}$ for all $\mathfrak{b}\in \mathrm{X}(R'),$ we have     
$\mathrm{ker}(f)\subseteq f\inv(\mathfrak{b}),$ or, in other words $f^*(\mathfrak{b})\in \mathrm{h}(\mathrm{ker}(f)).$ This implies that $\mathrm{im}(f^*)=\mathrm{h}(\mathrm{ker}(f)).$  
Since for all $\mathfrak{b}\in \mathrm{X}(R'),$ $$f(f^*(\mathfrak{b}))=f(f\inv(\mathfrak{b}))=\mathfrak{b}\cap \mathrm{im}(f)=\mathfrak{b}\cap R'=\mathfrak{b},$$ the map $f^*$ is injective. To show that $f^*$ is a closed map, first we observe that for any subbasic closed subset  $\mathrm{h}(\mathfrak{a})$ of  $\mathrm{X}(R')$, we have $$f^*(\mathrm{h}(\mathfrak{a}))=  f\inv(\mathrm{h}(\mathfrak{a}))=f\inv\{ \mathfrak{i}'\in \mathrm{X}(R')\mid \mathfrak{a}\subseteq   \mathfrak{i}'\}=\mathrm{h}(f\inv(\mathfrak{a})).$$   Now if $C$ is a closed subset of $\mathrm{X}(R')$ and $C=\bigcap_{ \alpha \in \Omega} (\bigcup_{ i \,= 1}^{ n_{\alpha}} \mathrm{h}(\mathfrak{a}_{ i\alpha})),$ then $$f^*(C)=f\inv \left(\bigcap_{ \alpha \in \Omega} \left(\bigcup_{ i = 1}^{ n_{\alpha}} \mathrm{h}(\mathfrak{a}_{ i\alpha})\right)\right)=\bigcap_{ \alpha \in \Omega} \bigcup_{ i = 1}^{ n_{\alpha}} \mathrm{h}(f\inv(\mathfrak{a}_{ i\alpha})),$$ a closed subset of  $\mathrm{X}(R).$ Since $f^*$ is   continuous, we have the proof.

(iii). We first show that $\overline{f^*(\mathrm{h}(\mathfrak{b}))}=\mathrm{h}(f\inv(\mathfrak{b})),$ for all ideals $\mathfrak{b}\in R'.$ To this end, let $\mathfrak{s}\in f^*(\mathrm{h}(\mathfrak{b})).$ This implies $f(\mathfrak{s})\in \mathrm{h}(\mathfrak{b}),$ which means $\mathfrak{b}\subseteq f(\mathfrak{s}).$ In other words, $\mathfrak{s}\in \mathrm{h}(f\inv(\mathfrak{b})).$ The other inclusion follows from the fact that $f\inv(\mathrm{h}(\mathfrak{b}))=\mathrm{h}(f\inv(\mathfrak{b})).$ Since, $$\overline{f^*(\mathrm{X}(R'))}=f^*(\mathrm{h}(\mathfrak{o}))=\mathrm{h}(f\inv(\mathfrak{o}))=\mathrm{h}(\mathrm{ker}(f)),$$ a necessary and sufficient condition for $\mathrm{h}(\mathrm{ker}(f))$ to be equal to $\mathrm{X}(R))$ is $\mathrm{ker}(f)\subseteq \cap_{\mathfrak{s}\in \mathrm{X}(R)}\mathfrak{s}.$ 
 
(iv). The ring homomorphism $f \colon R\to R_{ S}$ defined by $f(r)=r/1$ induces a map $f^*\colon \mathrm{X}(R_{ S})\to \mathrm{X}(R)$ defined by $f^*(\mathfrak{a})=f\inv(\mathfrak{a}).$ We claim that  $f^*(\mathfrak{a})\cap S=\emptyset.$ If not, let $s\in f^*(\mathfrak{a})\cap S.$ Then $$f(s)\in f(f\inv(\mathfrak{a})\cap S)=f(f\inv(\mathfrak{a}))\cap f(S)=\mathfrak{a}\cap f(S),$$ and hence $f(s)\in \mathfrak{a}.$ Since $f(s)$ is a unit in $R_{ S},$ this implies $\mathfrak{a}=R_{ S},$ a contradiction. Therefore, $f^*$ is indeed a map from $\mathrm{X}(R_{ S})$ to $\mathrm{X}(R)_S.$ If $f^*(\mathfrak{a})=f^*(\mathfrak{b})$ for some $\mathfrak{a},$ $\mathfrak{b}\in \mathrm{X}(R_{ S}),$ then $$\mathfrak{a}=f(f\inv(\mathfrak{a}))=f(f\inv (\mathfrak{b}))=\mathfrak{b}$$ shows that $f^*$ is injective.   
The map    $f^*\colon\mathrm{X}(R_{ S})\to\mathrm{X}(R)\backslash S$ is continuous follows from (\ref{contxr}). Since $f^*(\mathrm{h}(\mathfrak{a}))=\mathrm{h}(f\inv(\mathfrak{a})),$ the map $f^*$ is also closed. Therefore, $f^*$ has the desired properties.    
\end{proof}    

Taking $f$ as a quotient map, we immediately obtain a result similar to $\mathrm{Spec}(R).$ 

\begin{corollary}
 Let $\mathrm{X}(R)$ be a spectrum satisfying the contraction-ideal property. Then for each ideal $\mathfrak{a}$ of $R$, the ideal space $\mathrm{X}(R/\mathfrak{a})$ is homeomorphic to the closed subspace
$\mathrm{h}(\mathfrak{a})$ of $\mathrm{X}(R)$.   
\end{corollary}
 
From Proposition  \ref{conmap}, we get the well-known result that the ideal or structure  spaces $\mathrm{Spec}(R)$ and $\mathrm{Spec}(R/\sqrt{\mathfrak{o}})$ are canonically homeomorphic, and for  $f^*(\mathrm{Spec}(R'))$ to be dense in $\mathrm{Spec}(R),$ it is necessary and sufficient that $\mathrm{ker}(f)\subseteq \mathrm{h}(\mathfrak{o}).$ 

\subsection{Quasi-compactness and separation axioms}\label{sepax}

It is well known that $\mathrm{Spec}(R)$ is always quasi-compact, while, on the other hand, $\mathrm{Min}(R)$ is not always quasi-compact. Incidentally, the quasi-compactness of $\mathrm{Min}(R)$ for rings with no nonzero nilpotent elements is fully characterised in \cite{H88}. In this subsection we shall, among other things, find a condition that is not stringent which is sufficient for any ideal space $\mathrm{X}(R)$ to be quasi-compact.
	  
Also covered in this subsection is the identification of some irreducible subbasic closed sets of ideal spaces;  a result which enables us  to characterise when an ideal space is a $T_1$-space. For an arbitrary $\mathrm{X}(R)$, a subbasic closed set need not contain the element associated with it. We show that each does precisely when $\mathrm{X}(R)$ is sober.

\begin{definition}\label{fulle}
 We say an ideal space has the \emph{partition of unity} property if for every $\mathfrak{a}\in \mathrm{Idl}(R),$ $\mathrm{h}(\mathfrak{a})=\emptyset$ implies $\mathfrak{a}=R.$
\end{definition} 

\begin{proposition}
 For a structure space $\mathrm{X}(R)$ to have the partition of unity property, it is necessary and sufficient that $\mathrm{X}(R)$ contains all maximal ideals of $R$.
\end{proposition}
	
\begin{proof}
\textit{Necessity}.  Suppose $\mathrm{X}(R)$ has the partition of unity property. Let $\mathfrak{m}$ be a maximal ideal of $R$. If $\mathfrak{m}$ were not in $\mathrm{X}(R)$, then we would have $\mathrm{h}(\mathfrak{m}) = \emptyset$, and since $\mathfrak{m} \neq R$, we would have a contradiction.

\textit{Sufficiency}. Suppose $\mathrm{X}(R)$ contains all maximal ideals of $R$. Consider an ideal $\mathfrak{a}$ of $R$ such that $\mathrm{h}(\mathfrak{a}) = \emptyset$. This means that there is no proper ideal of $R$ belonging to $\mathrm{X}(R)$ containing $\mathfrak{a}$, and this forces $\mathfrak{a}$ to be $R$ because every  proper ideal of $R$ is contained in some maximal ideal.
\end{proof} 

\begin{theorem}\label{comp} 
 If an ideal space $\mathrm{X}(R)$ has the partition of unity property, then $\mathrm{X}(R)$ is quasi-compact. 
\end{theorem} 
  
\begin{proof}   
Let  $\{C_{ \alpha}\}_{\alpha \in \Lambda}$ be a family of subbasic closed sets of an ideal space $\mathrm{X}(R)$   such that $\bigcap_{\alpha\in \Lambda}C_{ \alpha}=\emptyset.$ Let $\{\mathfrak{a}_{ \alpha}\}_{\alpha \in \Lambda}$ be a family of ideals of $\mathrm{Idl}(R)$ such  that $\forall \alpha \in \Lambda,$  $C_{ \alpha}=\mathrm{h}(\mathfrak{a}_{ \alpha}).$  Since $\bigcap_{\alpha \in \Lambda}\mathrm{h}(\mathfrak{a}_{ \alpha})=\mathrm{h}\left(\sum_{\alpha \in \Lambda}\mathfrak{a}_{ \alpha}\right),$ we get  $\mathrm{h}\left(\sum_{\alpha \in \Lambda}\mathfrak{a}_{ \alpha}\right)=\emptyset,$ and that by Definition  \ref{fulle} gives $ \sum_{\alpha \in \Lambda}\mathfrak{a}_{ \alpha}=R.$ Then, in particular, we obtain $1=\sum_{\alpha_i\in \Lambda}a_{ \alpha_i},$ where $a_{ \alpha_i}\in \mathfrak{a}_{\alpha_i}$ and $a_{ \alpha_i}\neq 0$ for $i=1, \ldots, n$. This implies    $R=\sum_{  i \, =1}^{ n}\mathfrak{a}_{\alpha_i}.$ Therefore,   $\bigcap_{ i\,=1}^{ n}C_{ \alpha_i}=\emptyset,$ and hence by Alexander subbase theorem, $\mathrm{X}(R)$ is quasi-compact.  
\end{proof}  

Since  for a topological space to be $T_0$, it is necessary and sufficient that its specialisation order is a partial order (see \cite{P94}), any ideal space $\mathrm{X}(R)$ has the same property.  
 
\begin{theorem}
 Every ideal space $\mathrm{X}(R)$ is   $T_{ 0}.$ 
\label{ct0t1}  
\end{theorem}

Now here is the result that identifies some irreducible subbasic closed sets. As mentioned above, we will put it to good use further below when we deal with $T_1$ axiom and connectedness. 

\begin{theorem}\label{irrc}
 For every ideal space $\mathrm{X}(R)$, the subbasic closed sets $\{\mathrm{h}(\mathfrak{a})\mid \mathfrak{a}\in \mathrm{h}(\mathfrak{a})\}$ are irreducible. 
\end{theorem} 
 
\begin{proof} 
We show that $\mathrm{h}(\mathfrak{a})=\overline{\{\mathfrak{a}\}}$ for every $\mathfrak{a}$ for which $\mathfrak{a}\in\mathrm{h}(\mathfrak{a})$, which will prove the lemma. 	Since $\overline{\{\mathfrak{a}\}}$ is the smallest closed set containing $\mathfrak{a}$, and since $\mathrm{h}(\mathfrak{a})$ is a closed set containing $\mathfrak{a}$, we immediately have  $\overline{\{\mathfrak{a}\}}\subseteq \mathrm{h}(\mathfrak{a})$. 
For the reverse inclusion, we first consider the (possible) case that $\overline{\{\mathfrak{a}\}}= \mathrm{X}(R)$. Then  from the string of containments
\[ 
\mathrm{X}(R)=\overline{\{\mathfrak{a}\}}\subseteq \mathrm{h}(\mathfrak{a})\subseteq \mathrm{X}(R),
\] 
we obtain $\mathrm{h}(\mathfrak{a})=\overline{\{\mathfrak{a}\}}$. Now suppose that $\overline{\{\mathfrak{a}\}}\neq \mathrm{X}(R)$. Being a closed set, $\overline{\{\mathfrak{a}\}}$ is expressible as an intersection of some finite unions of subbasic closed sets. We can therefore find an  index set,  $\Omega$, such that,  for each $\alpha\in\Omega$, there is a positive integer $n_{\alpha}$ and ideals $\mathfrak{a}_{\alpha 1},\dots, \mathfrak{a}_{\alpha n_\alpha}$ of $\mathrm{Idl}(R)$ such that 
$$
\overline{\{\mathfrak{a}\}}={\bigcap_{\alpha\in\Omega}}\left({\bigcup_{ i\,=1}^{ n_\alpha}}\mathrm{h}(\mathfrak{a}_{\alpha i})\right).
$$
  
Since 
$\overline{\{\mathfrak{a}\}}$ is not the whole space, without loss of generality, we assume that for each $\alpha$, ${\bigcup_{ i\,=1}^{ n_\alpha}}\mathrm{h}(\mathfrak{a}_{\alpha i})$ is non-empty. Therefore, for each $\alpha$, $\mathfrak{a}\in   {\bigcup_{ i\,=1}^{ n_\alpha}}\mathrm{h}(\mathfrak{a}_{\alpha i})$, from which we deduce that $\mathrm{h}(\mathfrak{a})\subseteq {\bigcup_{ i=1}^{ n_\alpha}}\mathrm{h}(\mathfrak{a}_{\alpha i})$, hence $\mathrm{h}(\mathfrak{a})\subseteq \overline{\{\mathfrak{a}\}}$. Consequently, $\mathrm{h}(\mathfrak{a})=\overline{\{\mathfrak{a}\}}$; an irreducible closed set.	
\end{proof}     
 
\begin{corollary}\label{spiir}
 Every non-empty subbasic closed subset of $\mathrm{Prp}(R)$ is irreducible.
\end{corollary} 

\begin{proof}
Since for every non-empty subbasic closed subset $\mathrm{h}(\mathfrak{a})$ of $\mathrm{Prp}(R)$, the ideal $\mathfrak{a}$ is also in $\mathrm{Prp}(R),$ the proof now follows from Theorem \ref{irrc}.  
\end{proof}

It is well known that for $\mathrm{Spec}(R)$ to be $T_1$, it is necessary and sufficient that $\mathrm{Spec}(R)$ coincides with $\mathrm{Max}(R)$. In case of an ideal space $\mathrm{X}(R)$ to be $T_1$, the condition is slightly different.
 
\begin{theorem}
 For an ideal space $\mathrm{X}(R)$ to be $T_1,$ it is necessary and sufficient that $\mathrm{X}(R)\subseteq\mathrm{Max}(R)$. 
\end{theorem}  

\begin{proof} 
{\em Necessity}. Let $\mathfrak{a}\in \mathrm{X}(R)$. Then $\mathfrak{a}\in\mathrm{h}(\mathfrak{a})$, and so, by Theorem \ref{irrc}, $\overline{\{\mathfrak{a}\}}=\mathrm{h}(\mathfrak{a})$. Let $\mathfrak{m}$ be a maximal ideal with $\mathfrak{a}\subseteq\mathfrak{m}$. Then   $\mathfrak{m}\in 	\mathrm{h}(\mathfrak{a})=\overline{\{\mathfrak{a}\}} = \{\mathfrak{a}\}$, the last part because $\mathrm{X}(R)$ is $T_{ 1}$. Therefore $\mathfrak{m}=\mathfrak{a}$, showing that $\mathrm{X}(R)\subseteq \mathrm{Max}(R)$.  

{\em Sufficiency}. In $\mathrm{Max}(R)$, $\mathrm{h}(\mathfrak{m})=\{\mathfrak{m}\}$ for every maximal ideal $\mathfrak{m}$, so that $\mathfrak{m}\in \mathrm{h}(\mathfrak{m})$, and hence, by Theorem \ref{irrc}, $\overline{\{\mathfrak{m}\}}=\{\mathfrak{m}\}$, showing that the space is $T_{ 1}$.
\end{proof} 

\begin{corollary}
 Let $R$ be a Noetherian ring. If $\mathrm{X}(R)$ is a discrete space then $R$ is Artinian.  
\end{corollary}

Although every Noetherian topological
space can be written as a finite union of non-empty irreducible closed subsets, 
but $\mathrm{Prp}(R)$ always has this property irrespective of $R$ being Noetherian and hence $\mathrm{Prp}(R)$ being Noetherian. This follows from the fact that $\mathrm{h}(\mathfrak{o})$ is an irreducible closed subset in $\mathrm{Prp}(R)$ (see Corollary \ref{spiir}). It is quite interesting to study ideals space of Noetherian rings, since they are a fundamental tool in algebraic geometry and algebraic number theory. Further investigations in this circle of ideas will be given in \cite{FGS22}, where, in particular, Notherian ideal spaces will be characterised.
 
If $C$ is an irreducible closed subset of  a topological space $Y$ and $\mathcal{S}$ is a closed subbase of $Y$, then it is well known (see \cite[Section 7.2, p.\,308]{H73}) that $C$ is the intersection of members of $\mathcal{S}.$ For an ideal space $\mathrm{X}(R)$, we get more. 
  
\begin{proposition}\label{ircs} 
 If $C$ is a non-empty irreducible closed subset of an ideal space $\mathrm{X}(R)$, then $C\in \mathcal{S}_{\mathrm{h}}.$ 
\end{proposition} 

Next we characterise sober spaces.
 
\begin{theorem}\label{sob}  
 For an ideal space $\mathrm{X}(R)$ to be sober, it is necessary and sufficient that every non-empty irreducible subbasic closed set $\mathrm{h}(\mathfrak{a})$  contains $\sqrt[\mathrm X]{\mathfrak a}.$ 
\end{theorem} 

\begin{proof}
\textit{Necessity}. Suppose $\mathrm{X}(R)$ is a sober space and $\mathrm{h}(\mathfrak{a})$ is a non-empty irreducible subbasic closed subset of $\mathrm{X}(R)$. Then there exists $\mathfrak{b}\in \mathrm{X}(R)$ such that 
$\mathrm{h}(\mathfrak{a})=\overline{\{\mathfrak{b}\}}=\mathrm{h}(\mathfrak{b}).$ Then by Proposition \ref{hrx}(\ref{axa}) we have $\mathfrak{b}=\sqrt[\mathrm X]{\mathfrak a}\in \mathrm{X}(R)$. 
 
\textit{Sufficiency}.  Let every non-empty irreducible subbasic  closed set $\mathrm{h}(\mathfrak{a})$ contains $\sqrt[\mathrm X]{\mathfrak a}$ and let $C$ be an irreducible closed subset of $\mathrm X(R)$. By definition, $C=\bigcap_{i\in I}E_i$, where $E_i$ is a finite union of sets of the type $\mathrm{h}(\mathfrak{a})$, for suitable ideals $\mathfrak a$ of $R$. Since $C$ is irreducible, for every $i\in I$ there exists an ideal $\mathfrak a_i$ of $R$ such that $C\subseteq \mathrm{h}(\mathfrak{a}_i)\subseteq E_i$ and thus, if $\mathfrak c= \sum_{i\in I}\mathfrak a_i$, then we have  $C=\bigcap_{i\in I}\mathrm{h}(\mathfrak a_i)=\mathrm{h}(\mathfrak c)=\mathrm{h}(\sqrt[\mathrm X]{\mathfrak c})$. By assumption, $\sqrt[\mathrm X]{\mathfrak c}\in \mathrm X(R)$ and thus $C=\overline{\{\sqrt[\mathrm X]{\mathfrak c}\}}$.  
\end{proof}


  
\begin{proposition}\label{sirrs} 
 The ideal space $\mathrm{Irs}(R)$ is sober.
\end{proposition}
  
\begin{proof} 
Since $\mathrm{Idl}(R)$ is a multiplicative lattice with product as intersection, every strongly irreducible ideal becomes a prime element of this lattice and hence by \cite[Lemma 2.6]{FFJ22}, $\mathrm{Irs}(R)$ is sober.  
\end{proof}
  
Since every non-empty irreducible closed subset $\mathrm{h}(\mathfrak{a})$ of $\mathrm{Prp}(R)$ contains $\mathfrak{a},$ the ideal space $\mathrm{Prp}(R)$ is also sober. From here, it is natural to ask which ideal spaces are spectral in the sense of \cite{H69}; and that will be considered in \cite{FGS22}.   

\subsection{Connectedness}\label{comconn}  
    
We wish to present a disconnectivity result that bears resemblance to the fact that if $\mathrm{Spec}(R)$ is disconnected, then the ring has a proper idempotent element (see \cite{B72}, Section 4.3, Corollary 2). Towards that end, we introduce the following notion of  disconnectivity.

\begin{definition}\label{sconn}
 We say a closed subbase $\mathcal{S}$ of a topological space $Y$ \emph{strongly disconnects} $Y$  if there exist two non-empty subsets $A,$ $B\in\mathcal{S}$ such that $Y=A\cup B$ and $A\cap B=\emptyset$.
\end{definition}  

It is clear that if some closed subbase strongly disconnects  a  topological space, then the space is disconnected. Also, if a space is disconnected, then some closed subbase (for instance the collection of all its closed subspaces) strongly disconnects it. But this does not seem to imply that every closed subbase strongly disconnects it. We however have the following.

\begin{theorem}\label{th1}
 If a quasi-compact space is disconnected, then every closed base that is closed under finite intersections strongly disconnects the space.
\end{theorem}

\begin{proof} 
Suppose that $Y$ is a quasi-compact disconnected space, and let $\mathcal{B}$ 	be a closed base for $Y$ such that $\mathcal{B}$ is closed under finite intersections. Since $Y$ is disconnected, there are non-empty closed sets $A$ and $B$ such that $Y=A\cup B$ and $A\cap B=\emptyset$. Find families $\{A_{ i}\mid i\in I\}\subseteq \mathcal{B}$ and $\{B_{ j}\mid j\in J\}\subseteq\mathcal{B}$ such that $A=\bigcap_{ i}\, A_{ i}$ and $B={\bigcap_{ j}}\,B_{ j}$. Then, by the infinite  distributive laws in the Boolean algebra of subsets of any set,  
$$Y =A\cup B= \left({\bigcap_{ i}}\,A_{ i}\right)\cup \left({\bigcap_{ j}}\,B_{ j}\right)= {\bigcap_{ i,j}}\left(A_{ i}\cup B_{ j}\right),
$$
 
which implies that $A_{ i}\cup B_{ j}=Y$ for any pair $(i,j)\in I\times J$. On the other hand, the equality $A\cap B=\emptyset$ implies that  $\left({\bigcap_{ i}}\,A_{ i}\right)\cap \left({\bigcap_{ j}}B_{ j} \right)=\emptyset$, and hence, since $Y$ is quasi-compact, there are finitely many indices $i_{ 1},\dots, i_{ n}$ in $I$  and finitely many indices $j_{ 1},\dots, j_{ m}$ in $J$ such that
$$
\left(A_{ i_1}\cap \dots \cap A_{ i_n}\right) \cap \left(B_{ j_1}\cap\dots\cap B_{ j_m}\right)=\emptyset.
$$
Put   
$
\hat{A}=A_{ i_1}\cap \dots \cap A_{ i_n}$ and  $\hat{B}=B_{ j_1}\cap\dots\cap B_{ j_m},
$
and observe that, by the hypothesis on $\mathcal{B}$, both $\hat{A}$ and $\hat{B}$ belong to $\mathcal{B}$. Furthermore, since neither $A$ nor $B$ is empty, neither  $\hat{A}$ nor $\hat{B}$ is empty because $A\subseteq\hat{A}$ and $B\subseteq \hat{B}$. Therefore $\hat{A}$ and $\hat{B}$ are non-empty members of $\mathcal{B}$ such that $\hat{A}\cap\hat B=\emptyset$ and $Y=\hat A\cup \hat B$; the latter because $Y=A\cup B$. It follows therefore that $\mathcal{B}$ strongly disconnects $Y$.
\end{proof}

This result naturally leads one to ask when the base $\mathcal{B}_{\mathrm{h}}$ generated by the  subbase $\mathcal{S}_{\mathrm{h}}$ is closed under binary intersections. It turns out that it always has this property. A closed set $A$ belongs to $\mathcal{B}_{\mathrm{h}}$ if and only if it is expressible as
$
A= \mathrm{h}(\mathfrak a_{ 1})\cup\dots\cup \mathrm{h}(\mathfrak a_{ n}),
$
for some finitely many ideals $\mathfrak{a}_{ 1}, \dots,\mathfrak{a}_{ n}$ of $R$. 

\begin{lemma}\label{lem1}
 For any ideal space $\mathrm{X}(R)$, the closed base $\mathcal{B}_{\mathrm{h}}$ is closed under binary intersections.
\end{lemma}

\begin{proof}
Let $A$ and $B$ be elements of $\mathcal{B}_{\mathrm{h}}$, and pick finitely many ideals $\mathfrak{a}_{ 1}, \dots,\mathfrak{a}_{ n}$ and  $\mathfrak{b}_{ 1}, \dots,\mathfrak{b}_{ m}$ of $R$ such that 
$
A= \mathrm{h}(\mathfrak{a}_{ 1})\cup\dots\cup \mathrm{h}(\mathfrak{a}_{ n})$
and $B= \mathrm{h}(\mathfrak{b}_{ 1})\cup\dots\cup \mathrm{h}(\mathfrak{b}_{ m}).
$
Then 
$$
A\cap B=  \left({\bigcup_{ i\,=1}^n}\mathrm{h}(\mathfrak{a}_{ i})\right)\cap \left({\bigcup_{ j\,=1}^m}\mathrm{h}(\mathfrak{b}_{ j})\right)= {\bigcup_{ i\,=1}^n}{\bigcup_{ j\,=1}^m}\left(\mathrm{h}(\mathfrak{a}_{ i})\cap\mathrm{h}(\mathfrak{b}_{ j})\right)= {\bigcup_{ i\,=1}^n}{\bigcup_{ j\,=1}^m}\mathrm{h}(\mathfrak{a}_{ i}+\mathfrak{b}_{ j}),
$$  
showing that $A\cap B\in\mathcal{B}_{\mathrm{hk}}$.	
\end{proof}

In light of Theorem~\ref{th1}, we have the following corollary.

\begin{corollary}\label{cor1}
 If an ideal space $\mathrm{X}(R)$ is quasi-compact, then  for $\mathrm{X}(R)$ to be disconnected, it is necessary and sufficient that $\mathcal{B}_{\mathrm{h}}$ strongly disconnects $\mathrm{X}(R)$. 
\end{corollary}   

Here is the result we alluded to above. 

\begin{proposition}\label{pr1}  
 Suppose $R$ has zero Jacobson radical and $\mathrm{X}(R)$ contains all maximal ideals of $R.$ If the subbase $\mathcal{S}_{\mathrm{h}}$ strongly  disconnects  $\mathrm{X}(R)$, then  $R$ has a non-trivial idempotent element.
\end{proposition}

\begin{proof}
Let $\mathfrak{a}$ and $\mathfrak{b}$ be ideals of $R$ such that
$
\mathrm{h}(\mathfrak{a})\cap\mathrm{h}(\mathfrak{b}) =\emptyset,$ $\mathrm{h}(\mathfrak{a})\cup\mathrm{h}(\mathfrak{b}) =\mathrm{X}(R),$ and $\mathrm{h}(\mathfrak{a})\ne\emptyset \ne \mathrm{h}(\mathfrak{b}). $
Since $\mathrm{h}(\mathfrak{a})\cap\mathrm{h}(\mathfrak{b})=\mathrm{h}(\mathfrak{a} +\mathfrak{b})$, we therefore have $\mathrm{h}(\mathfrak{a}+\mathfrak{b})=\emptyset$ and hence $\mathfrak{a}+\mathfrak{b} =R$ because $\mathrm{X}(R)$ contains all maximal ideals of $R$. On the other hand, since $\mathrm{h}$ is a decreasing map, 
$
\mathrm{h}(\mathfrak{a}\mathfrak{b})\supseteq \mathrm{h}(\mathfrak{a})\cup\mathrm{h}(\mathfrak{b})=\mathrm{X}(R),
$
which then implies that $\mathfrak{a}\mathfrak{b}$ is contained in every maximal ideal of $R$, and is therefore the zero ideal since $R$ has zero Jacobson radical. Note that the condition $\mathrm{h}(\mathfrak{a})\ne\emptyset \ne \mathrm{h}(\mathfrak{b})$ implies that neither $\mathfrak{a}$ nor $\mathfrak{b}$ is the entire ring. So the 
equality $\mathfrak{a}+\mathfrak{b}=R$ furnishes non-zero elements $a\in\mathfrak{a}$ and $b\in\mathfrak{b}$ such that $a+b=1$. Since  $ab=0$ as $ab\in\mathfrak{a}\mathfrak{b}=\mathfrak{o}$, we therefore have
$
a= a(a+b)=a^{ 2}+ab =a^{ 2},
$ 
showing that $a$ is a non-zero  idempotent element in $R$. Since $\mathfrak{a}\ne R$,  $a\ne 1$, and hence $a$ is a non-trivial idempotent element of $R$.
\end{proof}

\begin{remark}\label{rem1}
 The proof of the foregoing result actually achieves more than what is asserted in the statement of the proposition. Let us explain.  
\begin{enumerate}[\upshape (i)]
\item
Since $\mathfrak{a}\mathfrak{b}=\mathfrak{o}$, we have that $\mathfrak{a} \cap\mathfrak  b=\mathfrak{o}$ and so $\mathfrak{a}$ is a direct summand, and therefore  generated by an idempotent element. 
\item
Following from (i), let $e$ be an idempotent element such that $\mathfrak{a}=\langle e\rangle$ and $f$ be an idempotent element such that $\mathfrak{b}=\langle f\rangle$. Then find elements $r,s\in R$ such that $1=re +sf$. Since $ef=0$, we therefore have $e=re$ and $f=sf$, whence we deduce that $f=1-e$. 
\end{enumerate} 
\end{remark}

In light of the discussion in Remark~\ref{rem1},  the conclusion in the statement of the proposition says there is a non-trivial idempotent element $e\in R$ such that
$
\mathrm{h}(\langle e\rangle)\cap\mathrm{h}(\langle 1-e\rangle) =\emptyset,$ $\mathrm{h}(\mathfrak \langle e\rangle)\cup\mathrm{h}(\langle 1-e\rangle) =\mathrm{X}(R),$ and $\mathrm{h}(\mathfrak \langle e\rangle)\ne\emptyset \ne \mathrm{h}(\mathfrak \langle 1-e\rangle).
$
This observation enables us to give an example of a ring $R$   satisfying all the hypotheses in Proposition~\ref{pr1}, with a non-trivial idempotent element but not strongly disconnected by $\mathcal{S}_{\mathrm{h}}$. 

\begin{example}\label{ex1}
 Let $R=\mathds{Z}\times \mathds{Z}$. Then $R$ has zero Jacobson radical. The elements $(1, 0)$ and $(0,1)$ are non-trivial idempotent elements in $R$, and are the only non-trivial idempotent elements in $R$. The collection $\mathrm{Prp}(R)$ contains all maximal ideals of $R$. We show that $\mathrm{Prp}(R)$ is not strongly disconnected by its subbase $\mathcal{S}_{\mathrm{hk}}$. For brevity, write $\mathfrak{a}$ and $\mathfrak{b}$ for the ideals of $R$ generated by $(1,0)$ and $(0,1)$, respectively. By the discussion preceding the example, if $R$ were strongly disconnected by the subbase $\mathcal{S}_{\mathrm{hk}}$, we would have (among other things) $\mathrm{h}(\mathfrak{a}) \cup \mathrm{h}(\mathfrak{b})=\mathrm{Prp}(R)$.  But this is not the case because, for instance, the ideal generated by $(2,2)$ belongs  neither to $\mathrm{h}(\mathfrak{a})$ nor $\mathrm{h}(\mathfrak{b})$. Thus, the converse of Proposition~\ref{pr1} fails.
\end{example}

Another consequence of Theorem \ref{irrc} is the following sufficient condition for an ideal space to be connected. 
 
\begin{theorem}\label{conis}
 If a spectrum $\mathrm{X}(R)$ contains the zero ideal, then the ideal space $\mathrm{X}(R)$ is connected.
\end{theorem}  

\begin{proof}
We observe that  $\mathrm{X}(R)=\mathrm{h}(\mathfrak{o})$ and irreducibility implies connectedness. Now the claim follows from Theorem \ref{irrc}. 
\end{proof}

\begin{examples}
 Examples of connected ideal spaces are $\mathrm{Prp}(R)$, $\mathrm{Fgn}(R)$, $\mathrm{Prn}(R)$. In particular, $\mathrm{Prp}(\mathds{Z}\times \mathds{Z})$ considered in Example \ref{ex1} is also connected. If $R$ is an integral domain, then the ideal spaces $\mathrm{Spec}(R)$, $\mathrm{Spn}(R)$ are connected. 
\end{examples}

Note that the converse of Theorem \ref{conis} may not be true. A trivial example is  $\mathrm{Max}(R)$, with $R$ having only one maximal ideal and $R$ is not a field. A more non-trivial example is as follows. Recall that if a Tychonoff space $Y$ is compact, then $Y$ is homeomorphic to $\mathrm{Max}(C(Y)).$ So if we take
$Y$ to be the unit interval $[0,1],$ then $\mathrm{Max}(C(Y))$ is connected but the zero ideal is not a maximal ideal. 
 
\section*{Acknowledgement}
We are indebted to the anonymous referee for some valuable remarks and suggestions. The second author thanks Carmelo Antonio Finocchiaro for some fruitful suggestions.


\begin{thebibliography}{99}  
 
\bibitem{A08} A. Azizi, Strongly irreducible ideals, \emph{J. Aust. Math. Soc.}, \textbf{84} (2008), 145--154.

\bibitem{B53} R. L.
Blair, Ideal lattices and the structure of rings, \emph{Trans. the Amer. Math. Soc.}, \textbf{75} (1953),  136--153.  

\bibitem{B72} N. Bourbaki, \emph{Elements of mathematics:
Commutative
Algebra}, Addison-Wesley, Reading, MA, 1972.
 
\bibitem{C66}
E. \v{C}ech,  \emph{Topological spaces}, Int. Publ. John Wiley \& Sons., 1966.

\bibitem{DST19}  M. Dickmann, N. Schwartz, and M. Tressal,  \emph{Spectral spaces}, Cambridge Univ. Press, 2019. 

\bibitem{F49} L. Fuchs, \"{U}ber die Ideale arithmetischer Ringe, \emph{Comment. Math. Helv.}, \textbf{23} (1949), 334--341. 

\bibitem{FFJ22} A. Facchini, C. A. Finocchiaro, and
G. Janelidze, Abstractly constructed prime spectra, Algebra Univers., \textbf{83}(8) (2022), 1--38.

\bibitem{FFS16} C. A. Finocchiaro, M. Fontana, and D. Spirito, A topological version of Hilbert’s Nullstellensatz, J. Algebra, \textbf{461} (2016), 25--41.
  
\bibitem{FGS22} C. A. Finocchiaro, A. Goswami, and D. Spirito, A classification problem of spectral ideal spaces (in preparation). 

\bibitem{FS20} C. A. Finocchiaro and D. Spirito, Suprema in spectral spaces and the constructible closure, New York J. Math., \textbf{26} (2020), 1064--1092.

\bibitem{GK39} 
I. Gelfand and A. Kolmogoroff, On rings of continuous functions on topological
spaces, \textit{C. R. (Dokl.) Acad. Sci. URSS, n. Ser.}, \textbf{22}  (1939), 11--15. 

\bibitem{GS41} I. Gelfand and G. \v{S}ilov, 
\"{U}ber verschiedene Methoden der Einf\"{u}hrung der topologie in die menge der maximalen ideale eines normierten ringes,
\emph{Math. Sbornik},  \textbf{9}(51) (1941), 25--39. 
     
\bibitem{G et. al.} G. Gierz, K. H. Hofmann, K. Keimel, J. D. Lawson,
M. Mislove, and D. S. Scott. \emph{Continuous lattices and domains},
Cambridge Univ. Press, 2003.

\bibitem{GJ60} 
L. Gillman and M. Jerison, \emph{Rings of continuous functions}, D. Van Nostrand Company, Inc., 1960. 
  
\bibitem{G60} 
A. Grothendieck,   \emph{\'{E}l\'{e}ments de g\'{e}om\'{e}trie alg\'{e}brique I, Le langage des sch\'{e}mas}, Inst. Hautes Études Sci. Publ. Math., No. \textbf{4}., 1960.
 
\bibitem{H73} D. Harris, Universal compact $T_{\scriptscriptstyle 1}$ spaces,       
\emph{General Topology and Appl.}, \textbf{3} (1973), 291--318. 

\bibitem{H88} J. A. Huckaba, \emph{Commutative rings with zero divisors},
Marcel Dekker, Inc., New York, 1988.

\bibitem{HRR02} W. J. Heinzer, L. J. Ratliff Jr., and D. E. Rush, Strongly irreducible ideals of a commutative ring, \emph{J. Pure Appl. Algebra}, \textbf{166}(3) (2002), 267--275.    
   
\bibitem{HJ65}M. Henriksen and M. Jerison,  The space of minimal prime ideals of a commutative ring, \textit{Trans. Amer. Math. Soc.}, \textbf{115} (1965), 110--130.

\bibitem{H67} 
M. Hochster,  \emph{Prime ideal structure in commutative rings}, Princeton Univ., Princeton, N. J., 1967.

\bibitem{H69} 
\bysame, Prime ideal structure in commutative rings, \textit{Trans. Am. Math. Soc.}, \textbf{142} (1969), 43--60.

\bibitem{H71} \bysame, The minimal prime spectrum of a commutative ring, \emph{Canad. J. Math.}, \textbf{XXIII}(5)  (1971), 749--758.

\bibitem{J45} N. Jacobson,
A topology for the set of primitive ideals in an arbitrary ring, \textit{Proc.
Nat. Acad. Sei. U.S.A.}, \textbf{31} (1945), 333--338.
 
\bibitem{J56} \bysame, \emph{Structure of rings}, Amer. Math. Soc. Colloquium Publications,
vol. \textbf{37}, Providence, 1956.
    
\bibitem{L53}
L. H. Loomis, \emph{An Introduction to
abstract harmonic analysis}, D. Van Nostrand Company, Inc., 1953.

\bibitem{McC48} N. H. McCoy, \emph{Rings and ideals}, Mathematical Association of America, 1948.  

\bibitem{M49} \bysame, Prime ideals in general rings, \emph{Am. J. Math.}, \textbf{71} (1949), 823--833.  

\bibitem{M53} J. D. McKnight, Jr., \emph{On the characterisation of rings of functions}, Purdu doctoral thesis, 1953.

\bibitem{O16} B. Olberding, Topological aspects of irredundant intersections of ideals and valuation rings in \emph{Multiplicative ideal theory and factorization theory},
\textbf{170} (2016), 277--307. 

\bibitem{P94} H. A. Priestley, Intrinsic spectral topologies, \emph{Papers on general topology and applications},
(Flushing, NY, 1992), 78–95, Ann. New York Acad. Sci., 728, New York Acad. Sci., New York, 1994.

\bibitem{S16} N. Schwartz, Strongly irreducible ideals and truncated valuations,\emph{Comm. Algebra}, \textbf{44}(3) (2016), 1055--1087.  

\bibitem{S39}
G. \v{S}ilov,  Ideals and subrings of the rings of continuous functions, \textit{C. R. (Dokl.) Acad. Sci. URSS, n. Ser.},  
\textbf{22} (1939), 7--10.  
 
\bibitem{S37}     
M. H. Stone,   Applications of the theory of Boolean rings to general topology, \textit{Trans. 
Am. Math. Soc.}, \textbf{41} (1937), 375--481.
\end{thebibliography}
\end{document}